\documentclass[12pt]{amsart}

\usepackage{fullpage}
\usepackage{latexsym,mathrsfs}
\usepackage{amsthm,amssymb,amsfonts,amsmath,dsfont}
\usepackage{xcolor}
\usepackage{hyperref}

  \renewcommand{\Pr}{\mbox{\rm Pr}}
  \newcommand{\Exp}{{\mathbb{E}}}
  \newcommand{\E}{\mathbb{E}}
  \DeclareMathOperator{\gw}{GW}

  \newcommand{\R}{\mathbb{R}} 
  \newcommand{\N}{\mathbb{N}} 
  \newcommand{\Z}{\mathbb{Z}} 
  \newcommand{\pmset}[1]{\{-1,1\}^{#1}} 
  \newcommand{\bset}[1]{\{0,1\}^{#1}} 
  \DeclareMathOperator{\im}{im} 

  \newcommand{\cF}{\mathcal{F}}
  \newcommand{\cM}{\mathcal{M}}
  \newcommand{\cP}{\mathcal{P}}
  
  \newcommand{\st}{:\,} 
  \newcommand{\ie}{{i.e.}}  
   
  \newcommand{\eps}{\varepsilon}
  
  \newcommand{\ceil}[1]{\lceil{#1}\rceil}
  \newcommand{\floor}[1]{\lfloor{#1}\rfloor}

  \newcommand{\indic}{1}
  
  \newcommand{\norm}[1]{\left\| #1 \right\|}
  
  \newcommand{\beq}{\begin{equation}}
  \newcommand{\eeq}{\end{equation}}
  \newcommand{\beqn}{\begin{equation*}}
  \newcommand{\eeqn}{\end{equation*}}
  \newcommand{\beqr}{\begin{eqnarray}}
  \newcommand{\eeqr}{\end{eqnarray}}
  \newcommand{\beqrn}{\begin{eqnarray*}}
  \newcommand{\eeqrn}{\end{eqnarray*}}
  \newcommand{\bmline}{\begin{multline}}
  \newcommand{\emline}{\end{multline}}
  \newcommand{\bmlinen}{\begin{multline*}}
  \newcommand{\emlinen}{\end{multline*}}

  \theoremstyle{plain}
  \newtheorem{theorem}{Theorem}[section]
  \newtheorem{lemma}[theorem]{Lemma}
  \newtheorem{proposition}[theorem]{Proposition}
  
  \newtheorem{corollary}[theorem]{Corollary}
  
  \theoremstyle{definition}
  \newtheorem{definition}[theorem]{Definition}

  \theoremstyle{remark}
  \newtheorem{remark}[theorem]{Remark}
  \renewenvironment{proof}[1][]{
    	\begin{trivlist}
     	\item[\hspace{\labelsep}{\em\noindent Proof#1:\/}]}
     	{{\hfill$\Box$}
    	\end{trivlist}
  }
  
\newif\ifnotes\notesfalse

\ifnotes
\usepackage{color}
\definecolor{mygrey}{gray}{0.50}
\newcommand{\notename}[2]{{\textcolor{red}{\footnotesize{\bf (#1:} {#2}{\bf ) }}}}
\newcommand{\noteswarning}{{\begin{center} {\Large WARNING: NOTES ON}\end{center}}}

\else

\newcommand{\notename}[2]{{}}
\newcommand{\noteswarning}{{}}

\fi

\author[Jop Bri\"{e}t]{Jop Bri\"et}
\address{CWI, Science Park 123, 1098 XG Amsterdam, The Netherlands}
\email{j.briet@cwi.nl}
\thanks{J.B.\ was supported by VENI grant~639.071.409 and the Gravitation-grant NETWORKS-024.002.003 from the Netherlands Organisation for Scientific Research~(NWO)}

\author[Sivakanth Gopi]{Sivakanth Gopi}
\address{Department of Computer Science, Princeton University, Princeton, NJ 08540, USA}
\email{sgopi@cs.princeton.edu}
\thanks{S.G.\ was supported by NSF CAREER award 1451191 and NSF grant CCF-1523816}

\begin{document}
\title{Gaussian width bounds with applications to arithmetic progressions in random settings}

\noteswarning

\begin{abstract}
Motivated by a problem on random differences in Szemer\'edi's theorem and another problem on large deviations for arithmetic progressions in random sets, we prove upper bounds on the Gaussian width of special point sets in~$\R^k$. 
The point sets are formed by the image of the $n$-dimensional Boolean hypercube under a mapping $\psi:\mathbb{R}^n\to\mathbb{R}^k$, where each coordinate is a constant-degree multilinear polynomial with 0-1 coefficients. We show the following applications of our bounds.
Let $[\mathbb{Z}/N\mathbb{Z}]_p$ be the random subset of~$\mathbb{Z}/N\mathbb{Z}$ containing each element  independently with probability~$p$.\\
\begin{itemize}
\item A set $D\subseteq \mathbb{Z}/N\mathbb{Z}$ is \emph{$\ell$-intersective} if any dense subset of~$\mathbb{Z}/N\mathbb{Z}$ contains a proper {$(\ell+1)$}-term arithmetic progression with common difference in~$D$.
Our main result implies that $[\mathbb{Z}/N\mathbb{Z}]_p$ is $\ell$-intersective with probability~$1 - o(1)$ provided $p \geq \omega(N^{-\beta_\ell}\log N)$ for $\beta_\ell = (\lceil(\ell+1)/2\rceil)^{-1}$. This gives a polynomial improvement for all $\ell \ge 3$ of a previous bound due to Frantzikinakis, Lesigne and Wierdl, and reproves more directly the same improvement shown recently by the authors and Dvir (here we avoid the theories of locally decodable codes and quantum information).\\
\item Let $X_k$ be the number of $k$-term arithmetic progressions in~$[\mathbb{Z}/N\mathbb{Z}]_p$ and consider  the large deviation rate  $\rho_k(\delta) = \log\Pr[X_k \geq (1+\delta)\mathbb{E}X_k]$. 
We give quadratic improvements of the best-known range of~$p$ for which a highly precise estimate of~$\rho_k(\delta)$  due to Bhattacharya, Ganguly, Shao and Zhao is valid for all odd $k \geq 5$.
In particular, the estimate holds if $p \geq \omega(N^{-c_k}\log N)$ for $c_k = (6k\lceil(k-1)/2\rceil)^{-1}$. \\
\end{itemize}
We also discuss connections with locally decodable codes and the Banach-space notion of type for injective tensor products of $\ell_p$-spaces.
\end{abstract}
\maketitle

\section{Introduction}
\label{sec:intro}

The \emph{Gaussian width} of a point set $T\subseteq \R^k$ measures the expected maximum correlation between~$T$ and a standard Gaussian vector~$g = N(0,I_k)$, and is given by
\beqn
\gw(T)
=
\Exp\big[\sup_{x\in T}\langle x, g\rangle\big].
\eeqn
The terminology reflects the fact that the Gaussian width of a set is proportional to~$\sqrt{k}$ times its average width in a random direction.
While this quantity plays a central role in high-dimensional probability, it is notoriously hard to estimate in general; see for instance~\cite{Talagrand:2014} for an extensive discussion of this problem.

Our main result gives upper bounds on the Gaussian width of sets that appear naturally in the context of probabilistic combinatorics.
The relevant sets are given by the image of the $n$-dimensional Boolean hypercube under a certain polynomial mapping $\psi:\R^n\to\R^k$. 
In particular, we focus on the case where each coordinate $\psi_i:\R^n\to\R$ is a multilinear polynomial with 0-1 coefficients.
Say that a polynomial  has \emph{multiplicity}~$t$ if each of its variables has a nonzero exponent in at most~$t$ monomials in its support.\footnote{
Here and below, $\lesssim, \gtrsim$ denote upper and lower bounds up-to absolute constants and $\lesssim_\eps, \gtrsim_\eps$ denote upper and lower bounds up-to constants depending on a parameter~$\eps$.}


\begin{theorem}\label{thm:main}
Let $\psi:\R^n\to\R^k$ be a polynomial mapping such that each coordinate is multilinear, has 0-1 coefficients, and has degree at most~$d$ and multiplicity~$t$.
Then,
\begin{align*}
\gw\big( \psi(\bset{n})\big)
\lesssim_{d}
nt\,\sqrt{kn^{1 - \frac{1}{\ceil{d/2}}}\log n}. 
\end{align*}
\end{theorem}
The factor~$nt$ can be seen as a natural scaling due to the fact that each coordinate~$\psi_i$ maps the Boolean hypercube into~$[0, nt]$ (which follows from a handshaking lemma).
In the special case where~$\psi$ is linear, $\psi(x) = (\langle c_1,x\rangle, \dots,\langle c_k,x\rangle)$, for some $c_1,\dots,c_k\in \bset{n}$, the set $\psi(\bset{n})$ is easily seen to be contained in the set $T = \{(\langle c_i, y\rangle)_{i=1}^k \st \|y\|_{\ell_\infty} \leq 1\}$.
The Gaussian width of the former set is thus at most that of the latter, which in turn is at most
\beqn
\Exp\Big[\Big\|\sum_{i=1}^k g_ic_i\Big\|_{\ell_1}\Big] \lesssim n\sqrt{k},
\eeqn
as the sum is an $n$-dimensional Gaussian vector whose coordinates have variance at most~$k$.
Perhaps surprisingly, Theorem~\ref{thm:main} shows that if~$\psi$ is quadratic and has constant multiplicity, then the Gaussian width is at most a factor~$\sqrt{\log n}$ larger than the above upper bound.
This turns out to be an easy consequence of a 1974 random matrix inequality due to Tomczak--Jeagermann~\cite{Tomczak-Jaegermann:1974}, which also forms the basis for our proof of the higher-degree cases.
The proof of Theorem~\ref{thm:main} (given in Section~\ref{sec:proof}) proceeds in two steps: first we reduce to the case of homogeneous mappings of even degree, and then we reduce to the quadratic case.
The first step is the reason for the ceiling in~$\ceil{d/2}$ appearing in the exponent of~$n$ and it would be interesting to know if one can remove this ceiling; \ie, does the result hold with the exponent~$1 - 2/d$?
More generally, an exponent of the form~$o(1)$ for constant~$d$ would imply the truth of some unresolved conjectures on variants of Szemer\'edi's theorem, large deviations and coding theory (topics which we discuss below).
The link to coding theory also implies that the bound is optimal for $d = 2$ and that the smallest possible exponent is at least~$(\log\log n)^{2-o(1)}/\log n$ for $d = 3$ and $(\log\log n)^{r-o(1)}/\log n$ for $d = 2^r$, $r \geq 3$.
Finally, a close inspection of the proof of Theorem~\ref{thm:main} shows that it also holds for polynomials with non-negative integer coefficients, for a suitable change of the definition of multiplicity. In the following four subsections we discuss two applications of this result and links with error correcting codes and the Banach space notion of type.

\subsection{Random differences in Szemer\'edi's Theorem}

In 1975 Szemer\'edi~\cite{Szemeredi:1975} proved that any subset of the integers of positive upper density contains arbitrarily long arithmetic progressions, answering a famous open question of Erd\H{o}s and Tur\'an.
It is well known that this is equivalent to the assertion that for every positive integer~$k$ and any $\alpha \in (0,1)$, there exists an $N_0(k,\alpha)\in\N$ such that if $N \geq N_0(k,\alpha)$ and $A\subseteq \Z/N\Z$ is a set of size $|A| \geq \alpha N$, then~$A$ must contain a proper $k$-term arithmetic progression.
Certain refinements of Szemer\'edi's theorem concern sets $D\subseteq \N$ for which the theorem still holds true when the arithmetic progressions are required to have common difference from~$D$.
Such sets are usually referred to as intersective sets in number theory, or recurrent sets in ergodic theory.
More precisely, a set $D\subseteq \N$ is $\ell$-intersective (or $\ell$-recurrent) if  any set $A\subseteq\N$ of positive upper density
has an $(\ell+1)$-term arithmetic progression with common difference in~$D$.
%
Szemer\'edi's theorem then states that $\N$ is $\ell$-intersective for every $\ell\in \N$, but much smaller intersective sets exist.
For example, for any $t\in \N$, the set $\{1^t,2^t,3^t,\dots\}$ is $\ell$-intersective for every~$\ell$, which is a special case of more general results of S\'ark\"{o}zy~\cite{Sarkozy:1978a} when $\ell = 1$ and of Bergelson and Leibman~\cite{Bergelson:1996} for all $\ell \geq 1$.
The shifted primes $\{p - 1\st \text{$p$ is prime}\}$ and $\{p + 1\st \text{$p$ is prime}\}$ are also $\ell$-intersective for every~$\ell\in\N$, shown by S\'ark\"{o}zy~\cite{Sarkozy:1978b} when~$\ell = 1$ and in a more general setting by Wooley and Ziegler~\cite{Wooley:2012} for all~$\ell \geq 1$.

It is natural to ask at what density, random sets become $\ell$-intersective.
To simplify the discussion, we will look at the analogous question in $\Z/N\Z$. 
\begin{definition}
Let~$\ell$ be a positive integer and~$\alpha \in (0,1]$.
A subset $D\subseteq \Z/N\Z$ is $(\ell, \alpha)$-intersective if any subset $A\subseteq \Z/N\Z$ of size~$|A| \geq \alpha N$ contains a proper $(\ell+1)$-term arithmetic progression with common difference in~$D$.
\end{definition}
It was proved independently by Frantzikinakis et al.~\cite{Frantzikinakis:2012} and Christ~\cite{Christ:2011} that for $\beta_\ell = \frac{1}{2^{\ell-1}}$ and $p\geq\omega(N^{-\beta_\ell}\log N)$, the random set~$[\Z/N\Z]_p$ is $(\ell, \alpha)$-intersective with probability $1 - o(1)$, provided $N\geq N_1(\ell, \alpha)$. 
This was improved for all $\ell \geq 4$ in~\cite{FLW16}, where it was shown that the same result holds with $\beta_\ell = \frac{1}{{\ell+1}}$, though it was conjectured there that $\beta_\ell = 1$ suffices for all~$\ell \geq 1$. 
Based on Theorem~\ref{thm:main} we obtain the following result, which improves on these bounds for all~$\ell \geq 3$.

\begin{theorem}\label{thm:intersective}
For every $\ell\in \N$ and $\alpha \in (0,1)$, there exists an $N_1(\ell, \alpha)\in\N$ such that the following holds.
Let $N\geq N_1(\ell, \alpha)$ be an integer and let $$\beta_\ell = \frac{1}{\ceil{\frac{\ell+1}{2}}} \quad \text{and} \quad p \geq \omega(N^{-\beta_\ell}\log N).$$ 
Then, with probability $1 - o(1)$, the set $[\Z/N\Z]_p$ is $(\ell,\alpha)$-intersective.
\end{theorem}

\subsection{Large deviations for arithmetic progressions}
Let $H = (V,E)$ be a hypergraph over a finite vertex set~$V$ of cardinality~$N$ and for $p\in (0,1)$ denote by~$V_p$ the random binomial subset where each element of~$V$ appears independently of all others with probability~$p$.
Let~$X$ be the number of edges in~$H$ that are induced by~$V_p$. 
Important instances of the random variable~$X$ include the count of triangles in an Erd\H{o}s--R\'enyi random graph and the count of arithmetic progressions of a given length in the random set~$[\Z/N\Z]_p$.

The study of the asymptotic behavior of~$X$ when~$p = p(N)$ is allowed to depend on~$N$ and~$N$ grows to infinity motivates a large body of research in probabilistic combinatorics.
Of particular interest is the problem of determining the probability that $X$ significantly exceeds its expectation $\Pr[X \geq (1 + \delta)\Exp X]$ for $\delta > 0$, referred to as the \emph{upper tail}.
Despite the fact that standard probabilistic methods fail to give satisfactory bounds on the upper tail in general, advances were made recently for special instances, in particular for triangle counts~\cite{Lubetzky:2017} and general subgraph counts~\cite{Bhattacharya:2017}.
For more general hypergraphs, progress was made by Chatterjee and Dembo~\cite{Chatterjee:2016} using a novel nonlinear large deviation principle (LDP), which was improved by Eldan~\cite{Eldan:2016} shortly after.
The LDPs give precise estimates on the upper tail that are given in terms of a parameter~$\phi_p$ whose value is determined by the solution to a certain variational problem. 
The range of values of~$p$ for which these estimates are actually valid depends on the underlying hypergraph~$H$.
This splits the problem of estimating the upper tail into two sub-problems: 
(1)~determining for what range of~$p$ the estimate in terms of~$\phi_p$ holds true
and (2)~solving the variational problem to determine the value of~$\phi_p$.
The answer to problem (1) turns out to depend on the Gaussian width of a point set related to~$H$.

This approach was pursued in~\cite{Chatterjee:2016} to estimate the upper tail of the number of 3-term arithmetic progressions in $[\Z/N\Z]_p$, for which the authors solved problem (1). 
The case of longer APs, asking for the upper tail probability of the count~$X_k$ of $k$-term arithmetic progressions in $[\Z/N\Z]_p$, was recently treated by  Bhattacharya et al.~\cite{Bhattacharya:2016}. 
They
solved the variational problem~(2) for $N$ prime and gave bounds for the relevant Gaussian width towards solving problem~(1).
Based on this, they showed that if $k\geq 3$ and $\delta > 0$ are fixed and $p$ tends to zero sufficiently slowly as $N\to\infty$ along the primes, then
\beq\label{eq:uppertail}
\Pr[X_k \geq (1+\delta)\Exp X_k] = p^{(1 + o(1))\sqrt{\delta} p^{k/2}N}.
\eeq

Similar results were shown for the analogous problem over $\{1,\dots,N\}$ (in which case $N$ no longer needs to be prime), but we shall focus on the problem in $\Z/N\Z$  for ease of exposition.
The rate at which~$p$ is allowed to decay for~\eqref{eq:uppertail} to hold turns out to depend on Gaussian widths of the form featuring in Theorem~\ref{thm:main}.
The bounds proved in~\cite{Bhattacharya:2016} imply that~\eqref{eq:uppertail} holds provided
$p \geq N^{-c_k}(\log N)^{\eps_k}$ for
\begin{align*}
c_3 = \frac{1}{18},
\quad
c_4 = \frac{1}{48}
\quad
\text{and}
\quad
c_k = \frac{1}{6k(k-1)}
\quad
\text{for $k\geq 5$},
\end{align*}
and absolute constants~$\eps_k \in (0,\infty)$ depending only on~$k$.
However, the authors conjecture that a probability~$p$ slightly larger than~$N^{-1/(k-1)}$ suffices for all~$k$. 
Some support for this conjecture is given by a result of Warnke~\cite{Warnke:2016} showing that for all $p \geq (\log N/N)^{1/(k-1)}$, the logarithm of the upper tail (also referred to as the large deviation rate) of the $k$-AP count in $\{1,\dots,N\}_p$ is given by $\Theta_k(\sqrt{\delta} p^{k/2} N\log p)$, where the asymptotic notation hides constants depending only on~$k$.
Notice that~\eqref{eq:uppertail} is more accurate than this result in that it (almost) determines those constants, though currently for a more narrow range of~$p$.\footnote{The main motivation for finding such precise estimates of the upper tail probability is not so much the problem itself as it is to understand structure of the set~$[\Z/N\Z]_p$ conditioned on~$X_k$ being much larger than its expectation (see~\cite{Bhattacharya:2016}).}
Using Theorem~\ref{thm:main}, we widen the  range of~$p$ for which~\eqref{eq:uppertail} can be shown to hold for all $k\geq 5$. 

\begin{theorem}\label{thm:uppertail}
For every integer $k \geq 3$ and
\beqn
c_k = \frac{1}{6k\big\lceil\frac{k-1}{2}\big\rceil},
\eeqn
the estimate~\eqref{eq:uppertail} holds true, provided $p \geq N^{-c_k}(\log N)$ and~$N$ is prime.
\end{theorem}


%
\subsection{Locally decodable codes}

There is a close connection between the Gaussian widths considered in Theorem~\ref{thm:main} and special error-correcting codes called \emph{locally decodable codes} (LDCs).
A map $C:\bset{k}\to\bset{n}$ is a $q$-query LDC if for every $i\in[k]$ and $x\in\bset{k}$, the value~$x_i$ can be retrieved by reading at most~$q$ coordinates of the codeword~$C(x)$, even if the codeword is corrupted in a not too large (but possibly constant) fraction of coordinates.
A main open problem is to determine the smallest possible codeword length~$n$ as a function of the message length~$k$, when~$q$ is a fixed constant.
Currently this problem is settled only in the cases~$q = 1,2$~\cite{Katz:2000,Kerenidis:2004,Goldreich:2006} and remains wide open for the case~$q =3$.
We refer to the extensive survey~\cite{Yekhanin:2012} for more information on this problem.
A connection with Gaussian width was established by the authors and Dvir in~\cite{BDG:2017}, where we show that $q$-query LDCs from $\bset{\Omega(k)}$ to $\bset{O(n)}$ are equivalent to mappings $\psi:\R^n\to\R^k$ whose coordinates are degree-$q$, multiplicity-1 polynomials with 0-1 coefficients that are supported by $\Omega(n)$ monomials, and such that the set $\psi(\bset{n})$ has Gaussian width~$\Omega(k)$.
It was observed there that the best-known lower bounds on the length~$n = n(k)$ of $q$-query LDCs---proved using techniques from quantum information theory~\cite{Kerenidis:2004}---imply a slightly different but equivalent version of Theorem~\ref{thm:intersective} (see Section~\ref{sec:intersective}).
The proof of Theorem~\ref{thm:main} is based on ideas from~\cite{Kerenidis:2004}, but does not use quantum information theory.\footnote{Not surprisingly, the LDC lower bounds of~\cite{Kerenidis:2004} are also implied by Theorem~\ref{thm:main}.}

\subsection{Gaussian width bounds from type constants}
We observe that the Gaussian width in Theorem~\ref{thm:main} can be bounded in terms of type constants of certain Banach spaces. Unfortunately, we do not have good enough bounds on the type constants of the required spaces to improve Theorem~\ref{thm:main}. But we hope that this connection will motivate progress on understanding these spaces. 

A Banach space~$X$ is said to have (Rademacher) type~$p> 0$	 if there exists a constant $T< \infty$ such that for every~$k$ and $x_1,\dots,x_k\in X$,
\beq\label{eq:type}
\E_{\eps}\Big\|\sum_{i=1}^k \eps_ix_i\Big\|_X^p \le T^p\sum_{i=1}^k \norm{x_i}_X^p,
\eeq
where the expectation is over a uniformly random $\eps=(\eps_1,\dots,\eps_k)\in \{-1,1\}^k$.
The smallest~$T$ for which~\eqref{eq:type} holds is referred to as the type-$p$ constant of~$X$, denoted~$T_p(X)$.
Type, and its dual notion cotype, play an important role in Banach space theory as they are tightly linked to local geometric properties (we refer to~\cite{Lindenstrauss-Tzafriri2} and~\cite{Maurey:2003} for extensive surveys).
Some fundamental facts are as follows.
It follows from the triangle inequality that every Banach space has type~1 and from the Kahane--Khintchine inequality that no Banach space has type $p>2$.
The parallelogram law implies that Hilbert spaces have type~2.
An easy but important fact is that $\ell_1$ fails to have type~$p>1$.
Indeed, a famous result of Maurey and Pisier~\cite{Maurey:1973} asserts that a Banach space fails to have type $p>1$ if and only if it contains~$\ell_1$ uniformly.
Finite-dimensional Banach spaces have type-$p$ for all $p\in [1,2]$.

Of importance to Theorem~\ref{thm:main} are the actual type constants~$T_p(X)$ of a certain family of finite-dimensional Banach spaces.
Let $r_1,\dots,r_d\ge 1$ be such that $\sum_{i=1}^d \frac{1}{r_i}=1$ and let $\mathcal L_{r_1,\dots,r_d}^n$ be the space of $d$-linear forms on~$\R^n\times\cdots\times \R^n$ ($d$ times) endowed with the norm
\beqn
\norm{\Lambda} = \sup\Big\{\frac{|\Lambda(x_1,\dots,x_d)|}{\|x_1\|_{\ell_{r_1}}\cdots\|x_d\|_{\ell_{r_d}}} \st x_1,\dots,x_d\in \R^n\setminus\{0\}\Big\}.
\eeqn
This space is also known as the injective tensor product of $\ell_{s_1}^n,\dots,\ell_{s_d}^n$ for $r_i^{-1} + s_i^{-1} = 1$ and as such plays an important role in the theory of tensor products of Banach spaces~\cite{Rya02}.
The relevance of the type constants of this space to Theorem~\ref{thm:main} is captured by the following lemma, proved in Section~\ref{sec:type}.

\begin{lemma}\label{lem:type_constants}
Let $\psi:\R^n\to\R^k$ be a polynomial mapping such that each coordinate is multilinear and has 0-1 coefficients, degree at most~$d$ and multiplicity~$t$.
Then for any $r_1,\dots,r_d \geq 1$ such that $\sum_{i=1}^d \frac{1}{r_i} = 1$ and any $p\in [1,2]$,
\begin{align*}
\gw\big( \psi(\bset{n})\big)
\lesssim_d nt\, T_p(\mathcal L_{r_1,\dots,r_d}^n)\,k^{1/p}.
\end{align*}
\end{lemma}

Observe that the space $\mathcal L_{2,2}^n$ may be identified with the space of $n\times n$ matrices endowed with the spectral norm (or operator norm).
A key ingredient in the proof of Theorem~\ref{thm:main}, Theorem~\ref{thm:TJ} below, easily implies that the type-2 constant of this space is of order~$O(\sqrt{\log n})$.
A well-known lower bound of the same order follows for instance from the connection between Gaussian width and LDCs and a basic construction of a 2-query LDC known as the Hadamard code.
More generally, lower bounds on the type constants of~$\mathcal L_{r_1,\dots,r_d}^n$ are implied by $d$-query LDCs~\cite{BNR:2012, Briet:2016}.

\section{Proof of Theorem~\ref{thm:main}}
\label{sec:proof}
In this section we prove Theorem~\ref{thm:main}.
We begin by giving a high-level overview of the ideas.
The main tool we use is the following random matrix inequality, which is a special case of a non-commutative version of the Khintchine inequality due to Tomczak-Jaegermann~\cite[Theorem~3.1]{Tomczak-Jaegermann:1974}.
Let~$\langle\cdot,\cdot\rangle$ be the standard inner product on~$\R^N$ and denote by $B_2^N$ the Euclidean unit ball in~$\R^N$.
Given a matrix~$A\in \R^{N\times N}$, its operator norm (or spectral norm) is given by $\|A\| = \sup\{|\langle Ax,y\rangle| \st x,y\in B_2^N\}$.

\begin{theorem}[Tomczak-Jaegermann]\label{thm:TJ}
There exists an absolute constant~$C \in (0,\infty)$ such that the following holds.
Let~$A_1,\dots,A_k \in \R^{N\times N}$ be a collection of matrices and
let $g_1,\dots, g_k$ be independent Gaussian random variables with mean zero and variance~1.
Then,
\beqn
\Exp\Big[\Big\|\sum_{i=1}^k g_i A_i\Big\|\Big]
\leq 
C\sqrt{\log N}
\Big(\sum_{i=1}^k \|A_i\|^2\Big)^{1/2}.
\eeqn
\end{theorem}

This result already suffices to prove Theorem~\ref{thm:main} when the coordinate mappings~$\psi_i$ are quadratic forms, in which case there exist matrices $A_i\in \bset{n\times n}$ such that $\psi_i(x) = \langle A_ix,x\rangle$.
The assumption that each~$\psi_i$ has multiplicity~$t$ implies that each row and column of~$A_i$ has at most $t$ ones.
This in turn implies that $\|A_i\| \leq t$ by a Birkhoff–-von Neumann-type theorem.
Since each $x\in \bset{n}$ has Euclidean norm at most~$\sqrt{n}$, we get
\begin{align*}
\gw\big(\psi(\bset{n})\big)
&=
\Exp\Big[\max_{x\in\bset{n}}\sum_{i=1}^k g_i\langle A_ix,x\rangle\Big] \\
&=\Exp\Big[\max_{x\in \bset{n}}\Big\langle\Big(\sum_{i=1}^kg_iA_i\Big)x, x\Big\rangle\Big]\leq
n\Exp\Big[\Big\|\sum_{i=1}^k g_iA_i\Big\|\Big].
\end{align*}
By Theorem~\ref{thm:TJ}, the above is at most $Ctn\sqrt{k\log n}$.
\medskip

The general case is proved via a reduction to the above quadratic case and consists of two steps.
In the first step, we reduce to the case where each coordinate~$\psi_i$ is a homogeneous polynomial of degree~$2\ceil{d/2}$.
This is done in a straightforward way by adding at most~$dn$ variables in such a way so as to preserve the multiplicity.
The second step consists of a reduction to the quadratic case.
For this, it will be convenient to consider the hypergraphs associated with the monomial support of the coordinate mappings~$\psi_i$.

Recall that an $d$-hypergraph~$H= (V,E)$ consists of a vertex set~$V$ and a multiset~$E$, also denoted~$E(H)$, of subsets of~$V$ of size at most~$d$, called the edges.
A hypergraph is $d$-uniform if each edge has size exactly~$d$.
The degree of a vertex is the number of edges containing it and the degree of~$H$, denoted~$\Delta(H)$, is the maximum degree among its vertices. 
A \textit{matching} is a hypergraph where no two edges intersect.
Associate with a hypergraph $H = ([n], E)$, the multilinear polynomial $p_H\in \R[x_1,\dots,x_n]$ given by 
\beq\label{eq:pHdef}
p_H(x_1,\dots,x_n)=\sum_{e\in E} \prod_{i \in e} x_i.
\eeq
The multiplicity of~$p_H$ is then exactly the degree~$\Delta(H)$.
Clearly the coordinate mappings~$\psi_i$ of the form featuring in Theorem~\ref{thm:main} can be written as~$p_H$ for some $d$-hypergraph~$H$ of degree at most~$t$.
The reduction to the quadratic case is based on the following key lemma, in which for $x\in \R^n$ and $m\in \N$, the the $m$th tensor power~$x$ is defined as $x^{\otimes m} = (\prod_{i=1}^m x_{u_i})_{u\in[n]^m}$.

\begin{lemma}[Matrix lemma]\label{lem:main}
For every~$r\in\N$ there exist a $C_r,c_r \in (0,\infty)$ and $n_0(r)\in \N$ such that the following holds.
Let $n \geq n_0(r)$, $m = C_rn^{1 - 1/r}$ and $N = n^m$.
Let $H = ([n], E)$ be a $2r$-uniform hypergraph and let $p_H$ be the polynomial as in~\eqref{eq:pHdef}.
Then, there exists a matrix $A\in \R^{N\times N}$ such that $\|A\| \lesssim_r \Delta(H)$ and for every $x\in\pmset{n}$,
\beqn
p_H(x)
=
\frac{n}{c_r N}\langle Ax^{\otimes m}, x^{\otimes m}\rangle.
\eeqn
Moreover,~$A$ is the adjacency matrix of a graph (with possible parallel edges).
\end{lemma}

\begin{remark}
We do not know if the value of~$m$ in Lemma~\ref{lem:main} is optimal.
More generally, a better bound in Theorem~\ref{thm:main} would follow if, for any fixed hypergraphs $H_1,\dots,H_k$ as in the lemma, there exist $N \leq o(n^{m})$, matrices $A_1,\dots,A_k\in \R^{N\times N}$ with $\|A_i\| \lesssim_r \Delta(H_i)$ and a map $f:\pmset{n}\to \{y\in \R^N \st \|y\|_2 \leq \sqrt{N}\}$ such that $p_{H_i}(x) = (n/c_rN)\langle A_i f(x),f(x)\rangle$ for every $i\in[k]$ and $x\in \pmset{n}$.
\end{remark}
 
With this lemma in hand, the proof of Theorem~\ref{thm:main} is straightforward (see below).
The idea behind Lemma~\ref{lem:main} is to use decompositions into matchings and a generalization of the Birthday Paradox that says that for any $n$-vertex $2r$-matching, a random subset of $C_rn^{1-1/r}$ vertices contains~$r$ vertices of any fixed edge with probability  $c_r/n$.
To illustrate how this is used in the $r=2$ case, let~$H$ be a 4-matching, let $m = C_2\sqrt{n}$ and $N = n^m$.
It follows from the generalized Birthday Paradox that there are  $c_2N/n$ strings in~$[n]^m$ containing at least two elements of a given edge.
Now let~$G$ be the graph with vertex set~$[n]^m$ whose edges are the pairs~$\{u,v\}$ that \emph{cover} some edge in~$H$ and \emph{complement} each other, meaning:
there are indices $i,j\in[m]$ such that $\{u_i,u_j,v_i,v_j\}\in E(H)$ and $u_\ell = v_\ell$ for all $\ell\not\in\{i,j\}$.
The main observation is that for every edge $\{u,v\} \in E(G)$ that covers an edge $e\in E(H)$ and every $x\in \pmset{n}$, we have 
\beqn
(x^{\otimes m})_u(x^{\otimes m})_v = 
\prod_{\ell=1}^m x_{u_\ell}x_{v_\ell} 
=
x_{u_i}x_{u_j}x_{v_i}x_{v_j}
=
\prod_{w\in e}x_w.
\eeqn
It follows that, modulo the relations $x_1^2 = 1,\dots,x_n^2 = 1$, we have $p_G(x^{\otimes m}) = (c_2N/n)p_H(x)$.
The (appropriately scaled) adjacency matrix of~$G$  then satisfies the second criterion of the lemma, but it will have large norm if~$G$ has high degree. To obtain a matrix with the desired norm, we consider a pruned version of~$G$ in which we keep only edges that do not cover too many edges of~$H$ (at the cost of only a constant-factor decrease of the constant~$c_2$).
\medskip
\medskip

We now give the formal proof of Theorem~\ref{thm:main}.
The following simple proposition is used for the first step, in which we homogenize the polynomials.
Given two hypergraphs~$H,H'$, say that   $H'$  \emph{majorizes}~$H$ if $V(H)\subseteq V(H')$ and if for each edge $e\in E(H)$, there is a unique edge $e'\in E(H')$ such that $e\subseteq e'$.

\begin{proposition}\label{prop:homogeneous}
For any $n$-vertex $d$-hypergraph~$H$, there is a $d$-uniform hypergraph~$H'$ on~$dn$ vertices that majorizes~$H$ and satisfies $\Delta(H') = \Delta(H)$.
\end{proposition}

\begin{proof}
Let $t = \Delta(H)$.
It follows from the handshaking lemma that $|E(H)| \leq tn$.
Partition  $E(H) = \{E_1,\dots,E_n\}$ into~$n$ pairwise disjoint sets of size at most~$t$ each.
Add to~$V(H)$ pairwise disjoint sets $W_1,\dots,W_n$ of $d-1$ new vertices each.
For each $i\in [n]$, complete each edge $e\in E_i$ to a set of size~$d$ by adding vertices from~$W_i$ and let~$H'$ be the hypergraph thus obtained.
Observe that we have not increased the degree of the vertices in~$V(H)$.
Since each~$E_i$ has size at most~$t$, the new vertices in~$W_i$ also have degree at most~$t$ and therefore, $\Delta(H') =t$.
It is trivial to verify that~$H'$ satisfies the other desired properties.
\end{proof}
\medskip

\begin{proof}[ of Theorem~\ref{thm:main}]
Let $r = \ceil{d/2}$ and for each $i\in[k]$, let $H_i$ be the $d$-hypergraph of degree ~$t$ such that $\psi_i = p_{H_i}$, with~$p_{H_i}$ as in~\eqref{eq:pHdef}.
Assume that $n \geq n_0(r)$ for $n_0(r)$ as in Lemma~\ref{lem:main}.
 We start by reducing to the setting where each~$H_i$ is $2r$-uniform and of degree at most~$t$.
To this end, let $H_i' = ([n]\cup[(2r-1)n], E_i')$ be a $2r$-uniform hypergraph that majorizes~$H_i$ as in Proposition~\ref{prop:homogeneous}, which exists since any $d$-hypergraph is a $2r$-hypergraph.
Then, for each $e\in E(H_i)$, there is a unique set $f(e)\subseteq [(2r-1)n]$ such that $e\cup f(e)\in E(H_i')$.
It follows that 
\beqn
p_{H_i}(x) = 
\sum_{e\in E(H_i)}\prod_{i\in e} x_i
=
\sum_{e\in E(H_i)}\prod_{i\in e} x_i\prod_{j\in f(e)} 1
=
p_{H'_i}((x, {\bf 1})),
\eeqn 
where ${\bf 1}\in \R^{(2r-1)n}$ is the all-ones vector.
Hence, if we let $\psi':\R^{2rn}\to\R^k$ be the polynomial map whose coefficients are given by~$p_{H'_i}$, then
\beqn
\gw\big(\psi(\bset{n})\big) 
\leq
\gw\big(\psi'(\bset{2rn})\big).
\eeqn
Since the dependence of our claimed bound on the Gaussian width is polynomial in~$n$, the extra vertices will result in an extra factor depending only on~$d$.
It thus suffices to prove the theorem for the case where $H_1,\dots,H_k$ are $2r$-uniform.

Observe that since the polynomials $\psi_i$ are multilinear, the Gaussian width is bounded from above by replacing binary vectors with sign vectors.
In particular,
\beqn
\gw\big( \psi(\bset{n})\big) \leq 
\Exp \max\Big\{\sum_{i=1}^k g_ip_{H_i}(x)\st x\in \pmset{n}\Big\}.
\eeqn
Let $m = C_rn^{1 - 1/r}$ and $N = n^m$ and for each $i\in[k]$, let
 $A_i\in \R^{N\times N}$ be a matrix for $p_{H_i}$ as in Lemma~\ref{lem:main}.
Then, for every $x\in\pmset{n}$,
\begin{align*}
\sum_{i=1}^k g_ip_{H_i}(x)
&=
\frac{n}{c_rN}\sum_{i=1}^n g_i \langle A_ix^{\otimes m}, x^{\otimes m}\rangle
\leq
\frac{n}{c_r}\Big\|\sum_{i=1}^kg_iA_i\Big\|,
\end{align*}
where in the inequality we used that $x^{\otimes m}$ has Euclidean norm~$\sqrt{N}$.
Taking expectations, it then follows from Theorem~\ref{thm:TJ} that the Gaussian width of $\psi(\bset{n})$ is at most
\begin{align*}
\frac{n}{c_r}\Exp\Big[\Big\|\sum_{i=1}^k g_iA_i\Big\|\Big]
\lesssim \frac{n}{c_r}\sqrt{\log N}\Big(\sum_{i=1}^k\|A_i\|^2\Big)^{1/2}
\lesssim_{r} nt\, \sqrt{kn^{1 - 1/r}\log n},
\end{align*}
where in the second inequality we used that $\|A_i\| \leq O_{r}(t)$ for each $i\in[k]$.
\end{proof}

\section{Proof of the matrix lemma}
In this section we prove Lemma~\ref{lem:main}.
The starting point is a decomposition of a bounded-degree hypergraph into a small number of matchings.
For this, we use the following basic result on edge colorings.
The \textit{edge chromatic number} of a hypergraph~$H$, denoted by $\chi_E(H)$, is the minimum number of colors needed to color the edges of $H$ such that no two edges which intersect have the same color.  
Note that~$\chi_E(H)$ equals the smallest number of matchings into which~$E(H)$ can be partitioned. 

\begin{lemma}
\label{lem:chromatic_number}
Let $H$ be a $d$-hypergraph. Then, $$\Delta(H)\le \chi_E(H)\le d (\Delta(H)-1)+1.$$
\end{lemma}
\begin{proof}
Clearly $\chi_E(H)\ge \Delta(H)$ since edges containing  a maximum degree vertex should get different colors. To prove the upper bound, form a graph $G$ whose vertices are $E(H)$, and add edges between intersecting hypergraph edges. Then~$\chi_E(H)$ is equal to the vertex chromatic number of the graph $G$, which, by Brooks' Theorem, is at most $\Delta(G)+1$. Since an edge in~$H$ can intersect at most $d(\Delta(H)-1)$ other edges, $\Delta(G)\le d(\Delta(H)-1)$.
\end{proof}

To deal with matchings, we introduce the following definitions.
Let~$\mathcal M\subseteq{[n]\choose 2r}$ be a maximal  $2r$-matching of~$[n]$.
Let $s =200\cdot 4^r$.
Given a string $x\in\pmset{n}$ write its $m$-fold tensor product as
\beqn
x^{\otimes m} = \Big( \prod_{i=1}^m x_{f(i)} \Big)_{f:[m]\to[n]}.
\eeqn
Given a mapping $f:[m]\to [n]$ and set $S\in \mathcal M$, let
\beqn
\mu_S(f) = \sum_{T\in {S\choose r}}\prod_{i\in T} |f^{-1}(i)|.
\eeqn 
Note that this is a count of the $r$-subsets $I\subseteq [m]$  such that $|S\cap f(I)| = r$.
Denote
\beqn
\phi(f) = \sum_{S\in \mathcal M} \mu_S(f).
\eeqn
For $\ell\in \N$, say that $f$ is \emph{$\ell$-good} if $1\le \phi(f) \le \ell$.
Say that $g:[m]\to[n]$ \emph{complements}~$f$ if it satisfies the following two criteria:
\begin{enumerate}
\item There exists exactly one $I\in{[m]\choose r}$ such that $f(I)\cup g(I)\in\mathcal M$.
\item For all $i\in[m]\smallsetminus I$, we have $g(i) = f(i)$.
\end{enumerate}
If $g$ complements~$f$ then clearly the converse also holds.
Say that the complementary pair $(f,g)$ \emph{covers $S\in\mathcal M$} if $f(I)\cup g(I) = S$.
Observe that if $(f,g)$ covers~$S$, then for every $x\in \pmset{m}$, we have
\beq\label{eq:fgx}
(x^{\otimes m})_f(x^{\otimes m})_g
=
\prod_{i=1}^m x_{f(i)}x_{g(i)}
=
\prod_{j\in S}x_j.
\eeq
Define the set of ordered pairs
\beq\label{eq:Pdef}
\mathcal P = \big\{(f,g) \st \text{$f$ is $s$-good and $g$ complements~$f$}\big\}.
\eeq



\begin{proposition}\label{prop:equalcover}
Let~$\mathcal P$ be as in~\eqref{eq:Pdef}. Then, for every $S\in\mathcal M$, the number of pairs $(f,g)\in \mathcal P$ that cover~$S$ equals~$|\mathcal P|/|\mathcal M|$.
\end{proposition}

\begin{proof}
Fix distinct sets $S,T\in\mathcal M$ and let $\pi\in S_n$ be a permutation such that $\pi(S)=T, \pi(T)=S$ and $\pi(i)=i$ for all $i\notin S\cup T$. Let $\cP_S$ be the set of pairs $(f,g)\in \cP$ which cover~$S$ and define $\cP_T$ similarly. We claim that the map $\psi: (f,g)\mapsto (\pi\circ f, \pi \circ g)$ is an injective map from~$\cP_S$ to $\cP_T$. It follows that~$T$ is covered by at least as many pairs from~$\mathcal P$ as~$S$ is. Similarly, interchanging $S$ and $T$, the converse also holds. To prove the claim, note that if $(f,g)$ covers $S$, then $(\pi\circ f, \pi\circ g)$ covers~$T$. Moreover, $\phi(\pi\circ f) = \phi(f)$ because $\pi$ maps edges of the matching $\cM$ to edges of $\cM$. Thus $\psi(\cP_S)\subset \cP_T$.  Finally $\psi$ is injective because if $\pi\circ f=\pi \circ f'$ for some $f,f':[m]\to [n]$, then $f=f'$. Hence~$\mathcal P$ covers all~$S\in\mathcal M$ equally.
\end{proof}

\begin{proposition}\label{prop:phig}
For every $(f,g)\in\mathcal P$, we have that $g$ is $s^2$-good.
\end{proposition}

\begin{proof}
Let $S\in\mathcal M$ and $(f,g)\in\mathcal P$ be such that $(f,g)$ covers~$S$.
Consider the histograms $F,G: [n]\to \{0,1,\dots,m\}$ given by $F(i) = |f^{-1}(i)|$ and $G(i) = |g^{-1}(i)|$ for each $i\in[n]$.
Then~$F$ and $G$ differ only in~$S$.
In particular, there is an $r$-set $T\subseteq S$ such that $G(i) = F(i) + 1$ for each $i\in T$ and $G(i) = F(i) - 1$ for each $i\in S\smallsetminus T$.
Hence,
\begin{align*}
\mu_S(g) &= \sum_{T\in {S\choose r}} \prod_{i\in T} G(i)\\
&\leq \sum_{T\in {S\choose r}} \prod_{i\in T} \big(F(i)+1\big)\\
&\leq \sum_{T\in {S\choose r}} \Big(1 + 2^r\prod_{i\in T} F(i)\Big)\\
&\leq
4^r + 2^r\mu_S(f).
\end{align*}
For all other $S'\in\mathcal M$, we have $\mu_{S'}(g) = \mu_{S'}(f)$.
Moreover, $f$ must be $s$-good for $(f,g)$ to belong to~$\mathcal P$.
It follows that
\beqn
\phi(g) = \sum_{S'\in\mathcal M}\mu_{S'}(g) \leq 4^r +2^r\sum_{S'\in\mathcal M}\mu_{S'}(f) = 4^r + 2^r\phi(f) \leq s^2,
\eeqn
where in the last line we used the choice of~$s = 200\cdot 4^r$.
\end{proof}

\begin{lemma}[Generalized birthday paradox]\label{lem:bday}
For every $r\in \N$ there exists a $C_r\in (0,\infty)$ and an $n_0(r)\in \N$ such that the following holds.
Let~$h$ be a uniformly distributed random variable over the set of maps from~$[m]$ to~$[n]$.
Then, provided $n \geq n_0(r)$ and $m = C_rn^{1 - 1/r}$,
\beqn
\Pr\big[\text{$h$ is $s$-good}\big] \geq \frac{1}{2}.
\eeqn
\end{lemma}

We postpone the proof of Lemma~\ref{lem:bday} to Section~\ref{sec:bday}.

\begin{corollary}\label{cor:PA}
Let~$\mathcal P$ be as in~\eqref{eq:Pdef} and let
$A:[n]^m\times[n]^m\to\bset{}$ be its incidence matrix, that is $A(f,g)=1 \iff (f,g)\in \cP$.
Then, $|\mathcal P| \geq \Omega(N)$ and every row and every column of~$A$ has at most $s^2(r!)$ ones. 
\end{corollary}

\begin{proof}
The first claim follows from Lemma~\ref{lem:bday} and the fact that $|\mathcal P|$ is at least the number of~$s$-good mappings.
If~$h$ is $l$-good, then there are at most~$l(r!)$ mappings from $[m]\to [n]$ that complement~$h$.
Hence, every row of~$A$ has at most~$s(r!)$ ones and by Proposition~\ref{prop:phig}, every column of~$A$ has at most~$s^2(r!)$ ones.
\end{proof}

With this, we can now prove Lemma~\ref{lem:main}.

\begin{proof}[ of Lemma~\ref{lem:main}]
Let $t = \Delta(H)$.
By Lemma~\ref{lem:chromatic_number}, $H$ can be decomposed into $\chi_E(H)\le 2rt$ matchings, which we denote by $\cF_1,\dots,\cF_{\chi_E(H)}$. Complete each~$\mathcal F_i$ to a maximal family~$\mathcal M_i$ of disjoint $2r$-subsets of~$[n]$ in some arbitrary way.
For each~$\mathcal M_i$, let~$\mathcal P_i$ be as in~\eqref{eq:Pdef}
and let $A_i:[n]^m\times [n]^m\to\bset{n}$ be its incidence matrix.
Set to zero all the entries of~$A_i$ that correspond to a pair $(f,g)$ covering a set in $\mathcal M_i\smallsetminus \mathcal F_i$.
Let $B = A_1 + \dots + A_{\chi_E(H)}$ and $A =(B + B^{\mathsf T})$.
It follows from~\eqref{eq:fgx} and Proposition~\ref{prop:equalcover} that for each $x\in\pmset{n}$, we have
\beq\label{eq:xAxPMx}
\Big\langle \sum_{i=1}^{\chi_E(H)} (A_i + A_i^{\mathsf T})x^{\otimes  m},  x^{\otimes m}\Big\rangle
=
2\sum_{i=1}^{\chi_E(H)} \frac{|\mathcal P_i|}{|\mathcal M_i|} \sum_{S\in \mathcal F_i} \prod_{j\in S}x_i.
\eeq
Since all $\mathcal M_i$ are maximal, they have the same size, as do the $\mathcal P_i$.
Hence, by Corollary~\ref{cor:PA}, there exists a constant $c_r\in (0,1]$ such that the right-hand side of~\eqref{eq:xAxPMx} equals $(2c_rN/n) p_H(x)$.
Let~$G$ be the graph with adjacency matrix~$A$, allowing for parallel edges.
Then~$G$ has degree at most~$2ts^2(r!)$.
It follows from Lemma~\ref{lem:chromatic_number} that~$G$ can be partitioned into~$O_r(t)$ matchings.
Since the adjacency matrix of a matching has unit norm, we get that~$\|A\| \leq O_r(t)$.
\end{proof}

\section{Proof of the generalized birthday paradox.}\label{sec:bday}

For the proof of Lemma~\ref{lem:bday}, we use a standard Poisson approximation result for ``balls and bins'' problems \cite[Theorem~5.10]{Mitzenmacher:2005}.
A discrete Poisson random variable~$Y$ with expectation~$\mu$ is nonnegative, integer valued, and has probability density function
\beq\label{eq:poisson}
\Pr[Y = \ell] = \frac{e^{-\mu}\mu^\ell}{\ell!},
\quad\quad
\forall \ell=0,1,2,\dots
\eeq

\begin{proposition}\label{prop:poisson-sum}
If $X,Y$ are independent Poisson random variables with expectations~$\mu_X,\mu_Y$, respectively, then $X+Y$ is a Poisson random variable with expectation $\mu_X + \mu_Y$.
\end{proposition}

\begin{lemma}\label{lem:poisson2}
Let~$h$ be a uniformly distributed map from~$[m]$ to~$[n]$.
For each~$i\in[n]$, let~$X_i = |h^{-1}(i)|$ and
let~${\bf X} = (X_i)_{i\in [n]}$.
Let~${\bf Y} = (Y_i)_{i\in [n]}$ be a vector of independent Poisson random variables with expectation $m/n$.
Then, for any nonnegative function $\Phi:(\N\cup\{0\})^n\to \R_+$ such that~$\Exp[\Phi({\bf X})]$  decreases or increases monotonically with~$m$, we have
\beqn
\Exp[\Phi({\bf X})] \leq 2 \Exp[\Phi({\bf Y})].
\eeqn
\end{lemma}

%

\begin{proof}[ of Lemma~\ref{lem:bday}]
Let~$C_r> 0$ be a parameter depending only on $r$ to be set later.
Let $\mu = C_rm/n = C_rn^{-1/r}$ and assume that $n \geq n_0(r):=4(C_rr)^r$.
For $h$ a random map as in Lemma~\ref{lem:poisson2}, we begin by lower bounding the probability of the event that $\phi(h) \geq 1$.
Recall that this occurs if there exists an $S\in\mathcal M$ and an $r$-subset $T\in {S\choose r}$ such that $T\subseteq \im(h)$.
Let~${\bf X}$ be as in Lemma~\ref{lem:poisson2}.
Let $\psi:(\N\cup\{0\})^n\to\bset{}$ be the function
\beqn
\psi(x) = 
\prod_{S\in\mathcal M}\prod_{T\in {S\choose r}}\Big(1 - \prod_{i\in T}1_{\geq 1}(x_i)\Big).
\eeqn
Then $\psi({\bf X})=1$ if $\phi(h) = 0$ and~$\psi({\bf X})$ decreases monotonically with~$m$.
Hence, for~${\bf Y}$ a Poisson random vector as in Lemma~\ref{lem:poisson2}, we have
\begin{align}
\Pr[\phi(h)=0] &= \Exp[\psi({\bf X})]\nonumber\\
&\leq
2\Exp[\psi({\bf Y})]\nonumber\\
&=
2\prod_{S\in\mathcal M}\Exp\Big[\prod_{T\in {S\choose r}}\Big(1 - \prod_{i\in T}1_{\geq 1}(Y_i)\Big)\Big],\label{eq:prod_prob}
\end{align}
where in the last line we used the fact that since the sets~$S\in\mathcal M$ are disjoint, the random variables
\beqn
\prod_{T\in {S\choose r}}\Big(1 - \prod_{i\in T}1_{\geq 1}(Y_i)\Big)
\eeqn
are independent.
The random variables $1_{\geq 1}(Y_i)$, $i\in S$, are independent Bernoullis that are zero with probability $e^{-\mu}$.
The expectation in~\eqref{eq:prod_prob} equals the probability that these random variables form a string of Hamming weight strictly less than~$r$.
Using that $n\ge 4(C_rr)^r$ and the fact that $1-x\le \exp(-x)\le 1-x+x^2/2$ when $x>0$, this probability is at most
\beqn
1-\Pr[\forall i\in T\ 1_{\ge 1}(Y_i)=1]
=
1 - (1 - e^{-\mu})^r
\leq
1-(\mu(1-\mu/2))^r
\leq
1-\frac{C_r^r}{en}
\leq
\exp\Big(-\frac{C_r^r}{en}\Big)
\eeqn
where $T\subset S$ is some fixed subset of size $r$.
Hence, since~$\mathcal M$ is maximal, the above and \eqref{eq:prod_prob} give
\beq\label{eq:hzero}
\Pr[\phi(h) = 0] \leq 2\exp\left(-\frac{C_r^r|\mathcal M|}{en}\right) \leq 2\exp\left(-\frac{C_r^r\floor{n/r}}{en}\right) \leq 2\exp\left(-\frac{C_r^r}{2er}\right).
\eeq
Set~$C_r=(6er)^{1/r}$, then the above right-hand side is at most~$1/4$.
Next, we upper bound the probability that $\phi(h)\geq s=200\cdot 4^r$.
Define $\chi:(\N\cup\{0\})^n\to \R_+$ by
\beqn
\chi(x) = \sum_{S\in\mathcal M}\sum_{T\in {S\choose r}} \prod_{i\in T} x_i.
\eeqn
Then, $\phi(h) = \chi({\bf X})$.
Moreover, $\Exp[\chi({\bf X})]$ increases monotonically with~$m$.
It thus follows from Lemma~\ref{lem:poisson2} that
\begin{align*}
\Exp[\phi(h)] 
\leq
2\Exp[\chi({\bf Y})]
&=
2\sum_{S\in \mathcal M}\sum_{T\in {S\choose r}} \prod_{i\in T} \Exp[Y_i]\\
&\leq
2|\mathcal M|{2r\choose r}\Big(\frac{m}{n}\Big)^r
\leq 2\cdot \frac{n}{r} \cdot 4^r \cdot (6er)n^{-1}\leq
50\cdot 4^r.
\end{align*}
where in the second line we used the fact that the $Y_i$ are independent.
By Markov's inequality,
$\Pr[\phi(h) > 200\cdot 4^r] \leq \tfrac{1}{4}$.
With~\eqref{eq:hzero}, we get that~$h$ is $s$-good with probability at least $1/2$.
\end{proof}

\section{Random differences in Szemer\'edi's Theorem}
\label{sec:intersective} 
In this section we prove Theorem~\ref{thm:intersective}. 
We first consider a slightly different random model where we form a random multiset $D_k$ of size $k$ by repeatedly sampling a uniformly random element from $\Z/N\Z$.
We will need the following equivalent formulation of Szemer\'edi's Theorem due to Varnavides~\cite{Varnavides:1959} (see~\cite[Theorem~4.8]{Tao:2007} for this exact formulation).

\begin{proposition}
\label{prop:Szemeredi_thm}
For every $\ell\in \N, \alpha\in (0,1]$ there exists $N_1(\ell, \alpha), \epsilon(\ell, \alpha)$ such that for every $N\ge N_1(\ell, \alpha)$, the following holds. Every subset $A\subseteq \Z/N\Z$ of size at least $\alpha N$ contains an $\epsilon(\ell, \alpha)$-fraction of all $\ell+1$ term arithmetic progressions in $\Z/N\Z$, that is, 
$$\E_{x \in \Z/N\Z, y \in \Z/N\Z\smallsetminus\{0\}}[\indic_A(x)\indic_A(x+y)\dots\indic_A(x+\ell y)]\ge \epsilon(\ell, \alpha).$$
\end{proposition}

\begin{proposition}\label{prop:intersective_Dk}
For all $\ell\in \N, \alpha\in (0,1]$ there exists $N_1(\ell, \alpha)\in\N$ such that for every $N>N_1(\ell, \alpha)$ the following holds. 
Let $k\geq\omega(N^{1-1/\ceil{(\ell+1)/2}}\log N)$ and
let $D$ be a random multiset of size $k$ obtained by sampling~$k$ times independently and uniformly at random from $\Z/N\Z\setminus \{0\}$. Then, with probability $1-o(1)$,  every subset $A\subseteq \Z/N\Z$ of size at least ~$\alpha N$ contains a proper arithmetic progression of length $\ell+1$ with common difference in $D$.
\end{proposition}

\begin{proof}
We will arrive at a contradiction assuming that the statement is false.
Let $\Gamma = \Z/N\Z$.
For $f:\Gamma\to\R$ and $y\in \Gamma\setminus\{0\}$, define
\beqn
\phi_y(f)= \Exp_{x\in \Gamma}[f(x)f(x+y)\dots f(x+\ell y)],
\eeqn
which is a degree $\ell+1$ polynomial over the variables $(f(x))_{x\in \Gamma}$. 
For a multiset $S\subseteq \Gamma\setminus\{0\}$, define
\beqn
\Lambda_S(f) = \frac{1}{|S|}\sum_{y\in S}\phi_y(f).
\eeqn
If~$f = \indic_A$, then this counts the fraction of proper $(\ell+1)$-term APs with common difference in~$S$ that lie completely in $A$.
Note that $\Exp_D[\Lambda_D(f)] = \Lambda_{\Gamma\setminus\{0\}}(f)$.

Let $N_1(\ell, \alpha)$ and $\epsilon(\ell, \alpha)$ be as in Proposition~\ref{prop:Szemeredi_thm}. Suppose that with a constant probability, there is a subset $A\subseteq \Gamma$ of size at least~$\alpha N$ with no proper $(\ell+1)$-term APs whose common difference lies in $D$. 
Then,
\beqn
\Pr_D\Big[\inf_{A:|A|\ge \alpha N} \Lambda_D(\indic_A)=0\Big]=\Omega(1).
\eeqn
By Proposition~\ref{prop:Szemeredi_thm}, for every $A\subseteq \Gamma$ of size at least $\alpha N$, we have that $\Lambda_{\Gamma\setminus\{0\}}(\indic_A) \geq \epsilon$.
We are going to apply a standard symmetrization trick to establish a connection with Gaussian width.
Let~$D'$ be an independent copy of~$D$.
Then,
\begin{align*}
\epsilon &\lesssim \E_D\Big[\sup_{A:|A|\ge \alpha N} \left|\Lambda_D(\indic_A) - \Lambda_{\Gamma\setminus\{0\}}(\indic_A)\right|\Big]\\
&= \E_D\Big[\sup_{A:|A|\ge \alpha N} \big|\Lambda_D(\indic_A) - \Exp_{D'}[\Lambda_{D'}(\indic_A)]\big|\Big] \\
&\le  \E_{D,D'}\Big[\sup_{A:|A|\ge \alpha N} \big|\Lambda_D(\indic_A) - \Lambda_{D'}(\indic_A)\big|\Big]\\
&=
\Exp_{y_1,\dots,y_k,y_1',\dots,y_k'\in \Gamma\setminus\{0\}}
\Big[
\sup_{A:|A|\ge \alpha N} \Big|\frac{1}{k}\sum_{i=1}^k\phi_{y_i}(1_A) - \phi_{y_i'}(1_A)\Big|
\Big]
\end{align*}
Observe that for i.i.d.\ random $y,y'\in \Gamma\setminus\{0\}$, the random variable $\phi_{y}(1_A) - \phi_{y'}(1_A)$ is symmetric in the sense that it has the same distribution as its negation.
Let~$\sigma_1,\dots,\sigma_k$ be independent uniformly distributed~$\pmset{}$-valued random variables. 
Then it follows from the above that
\begin{align*}
\epsilon &\lesssim  \E_{y_1,\dots,y_k\ y_1',\dots,y_k'\in \Gamma\setminus\{0\}}\E_\sigma\Big[\sup_{A:|A|\ge \alpha N} \Big|\frac{1}{k} \sum_{i=1}^{k} \sigma_i\left(\phi_{y_i}(\indic_A) - \phi_{y_i'}(\indic_A)\right)\Big|\Big] \\
&\le  2\E_{y_1,\dots,y_k\in \Gamma\setminus\{0\}}\E_\sigma\Big[\sup_{A:|A|\ge \alpha N} \Big|\frac{1}{k} \sum_{i=1}^{k} \sigma_i\phi_{y_i}(\indic_A)\Big|\Big]. \\
\end{align*}
Let us fix $y_1,\dots,y_k\in \Gamma\setminus \{0\}$.  Each $\phi_{y_i}$ can be written as $\phi_{y_i}=N^{-1}p_{H_i}$ (as in \eqref{eq:pHdef}) where~$H_i$ is the hypergraph on $\Gamma$ whose edges are given by $(\ell+1)$ term arithmetic progressions with common difference~$y_i$. The maximum degree of $H_i$ is $O(\ell)$. This is because each such AP $(x+ty_i)_{0\le t\le \ell}$ intersects another AP $(x'+t'y_i)_{0\le t'\le \ell}$ iff $x-x'=(t'-t)y_i$; so there are only~$O(\ell)$ such~$x'$ for a given~$x$. 
Let $g_1,\dots,g_k$ be independent~$N(0,1)$ random variables.
Then we can bound
\begin{align*}
\E_\sigma\Big[\sup_{A:|A|\ge \alpha N} \Big|\frac{1}{k} \sum_{i=1}^{k} \sigma_i\phi_{y_i}(\indic_A)\Big|\Big] 
&\lesssim \frac{1}{k} \E_g\Big[\sup_{A} \Big| \sum_{i=1}^{k} g_i\phi_{y_i}(\indic_A)\Big|\Big] \\
&=\frac{1}{Nk} \E_g\Big[\sup_{A} \Big| \sum_{i=1}^{k} g_ip_{H_i}(\indic_A)\Big|\Big] \\
&\lesssim_\ell \frac{1}{k} \sqrt{kN^{1-1/\ceil{(\ell+1)/2}}\log N},
\end{align*}
where the last line follows directly from Theorem~\ref{thm:main}.
Thus we get $k\lesssim_\ell N^{1-1/\ceil{(\ell+1)/2}}\log N$ which is a contradiction.
\end{proof}

We will the need following simple fact that conditioning on a high probability event will not change the probability of any event by much.
\begin{lemma}\label{lem:conditioning}
Let $A,E$ be some events in some probability space. If $\Pr[E]\ge 1-\eps$ then $|\Pr[A|E]-\Pr[A]|\le 2\eps/(1-\eps)$.
\end{lemma}
\begin{proof}
\begin{align*}
|\Pr[A|E]-\Pr[A]|&=\left|\frac{\Pr[A\cap E]}{\Pr[E]}-\Pr[A]\right|=\left| \frac{1}{\Pr[E]}\left(\Pr[A]+\Pr[E]-\Pr[A\cup E]\right)-\Pr[A]\right|\\
&\le  \left|\Pr[A]\left(\frac{1}{\Pr[E]}-1\right)\right| + \left|1-\frac{\Pr[A\cup E]}{\Pr[E]}\right|
\le \frac{2\eps}{1-\eps}.
\end{align*}
\end{proof}

\begin{proof}[ of Theorem~\ref{thm:intersective}]
Let $D_k$ be a random subset of $\Z/N\Z\smallsetminus\{0\}$ of size at most $k$, formed by sampling a uniformly random element from $\Z/N\Z$ for $k$ times. Let $D_p=[\Z/N\Z \setminus \{0\}]_p$ be a random subset of $\Z/N\Z\setminus \{0\}$ formed by including each element with probability $p$ independently.
We claim that if $D_k$ is $\ell$-intersective with probability $1-o(1)$, then $D_p$ will also be $\ell$-intersective with probability $1-o(1)$ when $p=2k/N$ and $k=\omega_N(1)$. 

Let $p=2k/N$ and $k=\omega_N(1)$. Let $E$ be the event that $D_p$ has size at least $k$. By the Chernoff bound,
$$1-\Pr[E] \le \exp\left(-\mathrm{D_{KL}}\left(\frac{p}{2}||p\right)N\right) \le \exp(-\Omega(pN)) = o(1)$$ where $\mathrm{D_{KL}}$ is the Kullback-Leibler divergence. By Lemma~\ref{lem:conditioning}, conditioning on~$E$ changes the probability of $D_p$ being $\ell$-intersective by $o(1)$. Conditioned on $E$, the probability that~$D_p$ is $\ell$-intersective is at least the probability that $D_k$ is $\ell$-intersective. Indeed, both $D_p$ and $D_k$, after conditioning on a given size reduce to the uniform distribution over all subsets of that size. Proposition~\ref{prop:intersective_Dk} thus implies $D_p$ is $\ell$-intersective when $p=\omega(N^{-1/\ceil{(\ell+1)/2}}\log N)$.
\end{proof}

\section{Upper tails for arithmetic progressions in random sets}
Here we prove Theorem~\ref{thm:uppertail}.
Let $\Gamma = \Z/N\Z$.
In the following we identify maps from a set~$S$ to~$\R$ with vectors in $\R^S$.
For $f:\Gamma\to\R$, define
\beq\label{eq:LambdaDef}
\Lambda_k(f) 
=
\sum_{a,b\in \Gamma, b\ne 0} f({a})f({a+b})f(a+2b)\cdots f({a + (k-1)b}).
\eeq
Observe that for a subset $A\subseteq \Gamma$, we have that $\Lambda_k(1_A)$ counts the number of proper $k$-term arithmetic progressions in~$A$.
Moreover, $\Lambda_k$ is an $N$-variate polynomial of degree~$k$.
Recall that the gradient of a polynomial $p\in \R[x_1,\dots,x_n]$ is the mapping $\nabla p:\R^n\to\R^n$ whose $i$th coordinate is given by $(\nabla p)_i = (\partial p/\partial x_i)(x)$.
The proof of Theorem~\ref{thm:uppertail} follows from a simple corollary of Theorem~\ref{thm:main} and one of the main results of~\cite{Bhattacharya:2016}.
For the corollary, we consider polynomial mappings given by gradients of polynomials of the form~\eqref{eq:pHdef}.

\begin{corollary}\label{cor:gw_H}
Let~$n,t,d$
be positive integers.
Let~$H = ([n], E)$ be a $(d+1)$-hypergraph such that at most~$t$ edges are incident on any given pair of vertices.
Then, 
\beqn
\frac{1}{n}\gw\big((\nabla p_H)(\bset{n})\big) \lesssim_d  t n^{1 - \frac{1}{2\ceil{d/2}}}\sqrt{\log n}.
\eeqn
\end{corollary}

\begin{proof}
For each~$i\in[n]$ let $H_i = ([n], E_i)$ be the $d$-hypergraph with edge set
\beqn
E_i = \{e\setminus \{i\} \st e\in E(H) \text{ and } i\in e\}.
\eeqn
The claim now follows from Theorem~\ref{thm:main} as $p_{H_i} = (\nabla p_H)_i$ each~$H_i$ has degree at most~$t$.
\end{proof}


\begin{theorem}[Bhattacharya--Ganguly--Shao--Zhao]\label{thm:BGSZ}
Let~$k \geq 3$ be a fixed integer and let~$\sigma, \tau$ be positive real numbers such that
\beqn
\frac{1}{N}\gw\big(\nabla \Lambda_k(\bset{\Gamma})\big) \lesssim N^{1-\sigma}(\log N)^{\tau}.
\eeqn
Let $p \in (0,1)$ be bounded away from~1 and let $\delta > 0$ be such that $\delta = O(1)$ and
\beqn
\min\{\delta p^k, \delta^2 p\}
\gtrsim
N^{-\sigma/3}(\log N)^{1 + \tau/3}.
\eeqn
Then,
\beq\label{eq:BGSZ}
\log\Pr[\Lambda_k(\Gamma_p) \geq (1+\delta)\Exp\Lambda_k(\Gamma_p)]
=
-\big(1 + o(1)\big)\,\phi_p\big((1 + o(1))\delta\big).
\eeq
Moreover, provided $\delta p^kN^2\to \infty$ and $N$ is prime, we have
\beqn
\phi_p(\delta)
\asymp
N\, \min\{\sqrt{\delta}p^{k/2}\log(1/p), \delta^2p\}.
\eeqn
\end{theorem}

\begin{proof}[ of Theorem~\ref{thm:uppertail}]
Let~$H = (\Gamma, E)$ be the hypergraph whose edges are the (unordered) proper $k$-term arithmetic progressions in~$\Gamma$.
Then, accounting for the fact that $\Lambda_k$ distinguishes between the same progression with step~$b$ run forward from a point~$a$ or backward from $a + (k-1)b$ and since $N$ is prime, we have $2p_H = \Lambda_k$.
We claim that every pair of distinct vertices appears in $O(k^2)$ edges.
First note that~$H$ is 2-transitive, since for any two pairs of distinct vertices $(a,b), (c,d)$, the affine linear map $x\mapsto c(x-b)/(a-b)+d(x-a)/(b-a)$ sends $a$ to $c$, $b$ to $d$ and preserves progressions.
It follows that every pair of distinct vertices is contained in the same number of edges.
Since each edge contains ${k\choose 2}$ pairs, the claim follows by double-counting.
By Corollary~\ref{cor:gw_H}, we may thus set $\sigma = 1/(2\ceil{(k-1)/2})$ and $\tau = 1/2$ in Theorem~\ref{thm:BGSZ} and it follows that for constant~$\delta$, the estimate~\eqref{eq:BGSZ} holds if
\beqn
p^k\gtrsim
\min\{\delta p^k, \delta^2 p\}
\gtrsim
N^{-\frac{1}{6\ceil{(k-1)/2}}}(\log N)^{1 + 1/6}.
\eeqn
Taking $k$th roots now gives the claim.
\end{proof}

\section{Proof of Lemma~\ref{lem:type_constants}}
\label{sec:type}
In this section we give a proof Lemma~\ref{lem:type_constants}.
As explained in the proof of Theorem~\ref{thm:main}, it suffices to prove the statement when the coordinates of $\psi$ are given by $p_{H_i}$ (as in~\eqref{eq:pHdef}) for $d$-uniform hypergraphs $H_1,\dots,H_k$.
Let $\Lambda_{H_i}$ be a $d$-multilinear form such that $p_{H_i}(x)=\Lambda_{H_i}(x,x,\dots,x)$. 
Let $g=(g_1,\dots,g_k)$ be vector of independent standard Gaussians and $\eps=(\eps_1,\dots,\eps_k)$ be uniformly random in $\pmset{k}$. 
Then,
\begin{align*}
\gw\big( \psi(\bset{n})\big) &= \E_g\sup_{x\in \bset{n}} \left|\sum_{i=1}^k g_ip_{H_i}(x)\right| \\
&=\E_g\sup_{x\in \bset{n}} \left|\sum_{i=1}^k g_i\Lambda_{H_i}(x,\dots,x)\right|\\
&\le \E_gn^{\sum_{i=1}^k 1/r_i} \norm{\sum_{i=1}^k g_i\Lambda_{H_i}}\\
&= n\E_g\Exp_\eps\norm{\sum_{i=1}^k \eps_ig_i\Lambda_{H_i}},
\end{align*}
where in the last line we used that each~$g_i$ is symmetrically distributed, that is, $g_i$ and~$-g_i$ have the same distribution.
By Jensen's inequality, the above expectation over~$\eps$ is at most
\begin{align*}
\left(\Exp_\eps\norm{\sum_{i=1}^k \eps_ig_i\Lambda_{H_i}}^p
\right)^{1/p}
\leq
T_p(\mathcal L_{r_1,\dots,r_s}^n) \left(\sum_{i-1}^k\norm{g_i\Lambda_{H_i}}^p\right)^{1/p},
\end{align*}
where the inequality follows from the definition of the type-$p$ constant of $\mathcal L_{r_1,\dots,r_s}^n$.
Hence,
\begin{align*}
\gw\big( \psi(\bset{n})\big) 
&\le n \E_g\  T_p(\mathcal L_{r_1,\dots,r_s}^n) \left(\sum_{i=1}^k \norm{g_i\Lambda_{H_i}}^p\right)^{1/p}\\
& \le n T_p(\mathcal L_{r_1,\dots,r_s}^n) \E_g \norm{g}_{\ell_p} \max_i \norm{\Lambda_{H_i}}\\
& \le n T_p(\mathcal L_{r_1,\dots,r_s}^n)\,k^{1/p} \max_i \norm{\Lambda_{H_i}},
\end{align*}
where we used the fact that $\E_g \norm{g}_{\ell_p} \le (\sum_{i=1}^k\E_{g_i}|g_i|^p)^{1/p} \le k^{1/p} (\E_{g_1}|g_1|^2)^{1/2}=k^{1/p}$.
If~$H_i$ is a matching hypergraph, using H\"older's inequality, it is easy to see that $\norm{\Lambda_{H_i}} \le 1.$ If not, by Lemma~\ref{lem:chromatic_number}, we can decompose~$H_i$ into $d\Delta(H_i)$ matchings and use triangle inequality to conclude that $\norm{\Lambda_{H_i}} \le d\Delta(H_i)$ which gives the desired bound.


\subsection*{Acknowledgements}
We thank Sean Prendiville, Fernando Xuancheng Shao and Yufei Zhao for helpful discussions.
This work was in part carried out while the authors were visiting the Simons Institute during the Pseudorandomness program of Spring 2017 and we thank the institute and organizers for their hospitality.

\bibliographystyle{alpha}
\bibliography{references}

\end{document}